\documentclass[10pt]{amsart}
\usepackage{amsmath,amssymb}

\newtheorem{theorem}{Theorem}[section]
\newtheorem{lemma}[theorem]{Lemma}
\newtheorem{proposition}[theorem]{Proposition}

\theoremstyle{definition}
\newtheorem{definition}[theorem]{Definition}

\newtheorem{corollary}[theorem]{Corollary}
\theoremstyle{remark}

\numberwithin{equation}{section}

  \renewenvironment{thebibliography}[1]{
    \begin{oldthebibliography}{#1}
      \setlength{\parskip}{0ex}
      \setlength{\itemsep}{0ex}
  }
  {
    \end{oldthebibliography}
  }

\begin{document}

\title{On the linear stability of K\"ahler-Ricci solitons}

\author{Stuart J. Hall}

\address{Department of Mathematics, Imperial College, London, UK, SW7
2AZ.}

\email{stuart.hall06@imperial.ac.uk}

\thanks{This work forms part of the first author's
Ph.D thesis funded by the EPSRC. He would like to thank his advisor
Professor Simon Donaldson for his comments and encouragement during
the course of this work.}
\author{Thomas Murphy}
\address{School of Mathematical Sciences, University College Cork, Ireland.}
\email{tommy.murphy@ucc.ie}
\thanks{The second author was supported by an
IRCSET postgraduate fellowship.\\We would also like to thank
Professor Huai-Dong Cao for useful communications.}

\subjclass[2010]{Primary  53C44, Secondary 53C25.}

\date{August 21, 2010.}

\keywords{Ricci solitons, Perelman's $\nu$-functional, linear
stability}

\maketitle

\begin{abstract}
We show that K\"ahler-Ricci solitons with $\dim H^{(1,1)}(M)\ge 2$
are linearly unstable.  This extends the results of
Cao-Hamilton-Ilmanen in the K\"ahler-Einstein case.
\end{abstract}
\section{Introduction}
A Ricci soliton is a complete Riemannian metric $g$, a vector field
$X$ and a real number $c$ satisfying
$$Ric(g)+\mathcal{L}_{X}g- cg = 0.$$
The soliton is called shrinking, steady or expanding if $c >0, c=0$
or $c<0$ respectively.  If $X=\nabla f$ for a smooth function $f$
the soliton is said to be \textit{a gradient Ricci soliton} with
potential function $f$.  If $X=0$ we recover the definition of an
Einstein metric. Solitons are important in the theory of the Ricci
flow as they occur as fixed points of the flow, up to the action of
the diffeomorphism group. Perelman's groundbreaking observation
\cite{Per} was that solitons are the critical points of a certain
functional (the $\nu$-functional which we define below). An
interesting question is whether the functional is locally maximised
at these points (i.e. linearly stable) as this affects the behaviour
of the Ricci-flow nearby to a soliton. It is this question that we
take up in this article for a special class of gradient shrinking
solitons, namely \textit{K\"ahler-Ricci solitons}.  These are
solitons on Fano manifolds which have the additional property that
$g$ is a K\"ahler metric and the vector field $\nabla f$ is
holomorphic.
\\
The only known examples of compact Ricci solitons which are not
products or Einstein are the Koiso-Cao soliton \cite{Koiso}, \cite{Cao},
a $U(2)$-invariant soliton explicitly constructed on $\mathbb{CP}^{2}\sharp-\mathbb{CP}^{2}$
and a class of toric-K\"ahler metrics constructed on Fano manifolds by Wang and Zhu \cite{WangZhu}.
\\
The precise theorem we wish to prove is the following:
\begin{theorem}\label{main} If $(M,g, J)$ is a K\"ahler-Ricci Soliton with
 $\dim H^{(1,1)}(M) \ge 2$, then the soliton is linearly unstable.
\end{theorem}

\begin{corollary}\label{minor}
The Koiso-Cao soliton on $\mathbb{CP}^{2}\sharp-\mathbb{CP}^{2}$ and
the Wang-Zhu soliton on $\mathbb{CP}^{2}\sharp-2\mathbb{CP}^{2}$ are
both linearly unstable.
\end{corollary}

This result generalises the result of Cao-Hamilton-Ilmanen who gave
a simple argument in \cite{CHI} for K\"ahler-Einstein metrics
satisfying $\dim H^{(1,1)}(M)\ge2$. The generalisation to
K\"alher-Ricci solitons is also suggested in their paper. We begin
by recalling some facts about the $\nu$-functional, then proceed to
compute the second variation of $\nu$ and define linear stability
precisely.  We finally show how one can construct unstable
variations in a manner similar to \cite{CHI}.
\\
The question of linear stability for K\"ahler-Ricci solitons has
also been considered by Tian and Zhu \cite{TianZhu}. They prove that
if one considers variations in Kahler metrics in the fixed class
$c_{1}(M)$ then the $\nu$-energy is maximised at a Kahler-Ricci
soliton. There is also the recent work of Cao and Meng-Zhu
\cite{CMZ}
who also give a detailed calculation of the second variation of $\nu$. \\
\\

\section{The Perelman $\nu$-functional}
Throughout this paper $(M,g)$ will denote a closed Riemannian
manifold. Most of the material here is contained in some form in
\cite{Top} and \cite{Cao2} and of course, was originally taken from
\cite{Per}. Perelman defined the following functional on triples
$(g,f,\tau)$, where $g$ is a Riemannian metric, $f$ a smooth
function and $\tau>0$ a constant.
\begin{definition}
  Let $f$ be a smooth function on $M$
  and $\tau > 0$ a real number. Then the $W$-functional is given by
  $$ W(g,f,\tau) = \int_{M} [\tau (R + |\nabla f|^{2}) + f - n] (4 \pi
     \tau)^{- \frac{n}{2}} e^{- f} d V_g, $$
  where $R$ is the scalar curvature of the metric $g$.
\end{definition}
Some authors \cite{TianZhu} absorb the constant $\tau$ into the
other two terms
 and it is not hard to see that
$\mathcal{W}(g,f,\tau)=\mathcal{W}(\tau^{-1}g,f,1)$.  The
$\mathcal{W}$-functional is also invariant under diffeomorphisms
i.e.
 $\mathcal{W}(g,f,\tau)=\mathcal{W}(\phi^{\ast}g,\phi^{\ast}f,\tau)$
for any diffeomorphism $\phi:M\rightarrow M$. Fix a compatibility
condition for the triple $(g,f,\tau)$ by requiring
$$\int_{M}\frac{e^{-f}}{(4\pi\tau)^{\frac{n}{2}}}dV_{g}=1.$$
This leads to the definition of the $\nu$-functional:
\begin{definition}
Let  $f$ be a smooth function on $M$
  and $\tau > 0$ a real number. Then the $W$-functional is given by
  $$ \nu(g) =\inf\{\mathcal{W}(g,f,\tau) : (g,f,\tau) \text{ is compatible}\}. $$
\end{definition}
We will not comment on the existence theory except to say that for
fixed $g$ there exists a $\tau>0$ and smooth $f$ that attain the
infimum in the above definition. The pair $(f,\tau)$ satisfy the
equations
$$\tau(-2\Delta f+|\nabla f|^{2}-R)-f+n+\nu=0 \text{, and } (4\pi\tau)^{-\frac{n}{2}}\int_{M}fe^{-f}dV_{g}=\frac{n}{2}+ \nu.$$
The first important result about the $\nu$-function is the
following:
\begin{theorem}[Perelman \cite{Per}]
Let $g(t)$ be a family of Riemannian metrics on $M$ evolving via the
Ricci-flow. Then $\nu(g(t))$ is montone increasing, unless $g$ is a
Ricci-soliton in which case $\nu$ is stationary.
\end{theorem}
We also record the first variation of $\nu$ which makes it clear
that
 stationary points of the functional are shrinking gradient Ricci solitons.
\begin{theorem}[Perelman \cite{Per}]
The first variation of $\nu$ in the direction $h$,
$\mathcal{D}_{g}\nu(h)$ is given by
$$-(4\pi\tau)^{-\frac{n}{2}}\int_{M}\langle \tau(Ric(g) + Hess(f))-\frac{1}{2}g,h\rangle e^{-f}dV_{g}.$$
\end{theorem}

\section{The second variation of $\nu$}
The aim of this section give a self-contained proof of the second
variation formula
 of $\nu$ at a Ricci soliton with potential function $f$. This is not new and has been
known to experts for some time but we include it here for
completeness. We work with the scaled $L^{2}$-inner product on
tensors
$$\langle \cdot,\cdot \rangle_{f}  = \int_{M}\langle \cdot,\cdot \rangle e^{-f}dV_{g}.$$
This inner product is adapted to the following operators
$$div_{f}(h) = e^{f}div(e^{-f}h) \text{ and } \Delta_{f}(h) =\Delta(h)-(\nabla_{\nabla f}h),$$
in the sense that
$$\int_{M} div_{f}(\alpha)e^{-f}dV_{g}=0 \text{ and } \int_{M} \Delta_{f}(F)e^{-f}dV_{g}=0$$
for any one-form $\alpha$ and function $F$. Obviously these reduce
to the usual divergence and Laplacian when the metric is Einstein.
The operator $\Delta_{f}$ is often referred to as the
Bakry-\'{E}mery Laplacian.  The sign convention for the curvature
tensor we adopt is
$$R (X, Y) Z = \nabla_Y \nabla_X Z - \nabla_X \nabla_Y Z + \nabla_{[X, Y]} Z$$
and $Rm(X, Y, W, Z) = g(R(X,Y)W,Z)$. In index notation we have
$R(\partial_{i},\partial_{j},\partial_{k},\partial_{l})=Rm_{ijkl}$.
For $h\in s^2(TM^{\ast})$, define the symmetric curvature operator
$Rm(h,\cdot)\in s^{2}(TM^{\ast})$ by
$$Rm(h, \cdot)_{ij}=R_{kilj}h^{kl}.$$
We will also need the curvature operator on 2-forms, usually denoted
by $\mathcal{R}$
$$\mathcal{R}:\Lambda^{2}(M)\rightarrow \Lambda^{2}(M), \ \ \mathcal{R}(\sigma)_{ij}=Rm_{ijkl}\sigma^{kl}.$$
The convention for divergence we adopt is $div(h)=tr_{12}(\nabla
h)$.  The reader should note that this definition is the opposite
sign to the divergence operator considered in \cite{Top}. When
restricted to forms we also have the codifferential $\delta$ which,
with this convention, satisfies $\delta(\sigma)=-div(\sigma)$.
 If we denote by $div_{f}^{\ast}$ the adjoint to  $div_{f}$ with respect to the scaled
  $L^{2}$-inner product $\langle \cdot, \cdot \rangle_{f}$, then, as remarked in \cite{Cao3}, $div_{f}^{\ast}=div^{\ast}$.
   Here $div^{\ast}$ is the adjoint to $div$ with respect to the usual inner product on tensors.

\begin{theorem}[Cao-Hamilton-Ilmanen]\label{T1}
Let $g$ be an Ricci-Soliton with potential function $f$ satisfying
$Ric(g)+Hess(f)=\frac{1}{2\tau}g$. For $h\in s^2(TM^{\ast})$,
consider variations  $g(s)=g+sh$. Then the second variation of the
$\nu$-energy is
$$\frac{d^2}{ds^2} \nu(g(s)) \bigg|_{s=0}=
\frac{\tau}{(4\pi\tau)^{\frac{n}{2}}}\int_{M}\langle h,Nh\rangle
e^{-f} dV_{g},$$ where $N$ is given by
$$N(h)=\frac{1}{2}\Delta_{f}(h)+Rm(h,\cdot)+div_{f}^{\ast}div_{f}(h)+\frac{1}{2}Hess(v_{h})-C(h,g){Ric}.$$
Here $v_h$ is the solution of
\begin{equation}\label{E1}
\Delta_{f}(v_{h})+\frac{1}{2\tau}v_{h}=div_{f}div_{f}(h)
\end{equation}
and $C(h,g)$ is a constant depending upon $h$ and $g$.
\end{theorem}

\textit{Remark}: This theorem was first stated in \cite{Cao3} but
with an error in the term $C(h,g)$. This has subsequently been
corrected in the recent work of Cao and Meng-Zhu \cite{CMZ}.  They
find that
$$C(h,g) = \frac{\int_{M}\langle Ric,h\rangle e^{-f}dV_{g}}{\int_{M}Re^{-f}dV_{g}}$$
where $R$ is the scalar curvature of $g$.\\

\begin{definition}
A soliton is linearly stable if the operator $N$ is non-positive
definite and linearly unstable otherwise.
\end{definition}
We now collect some formulae for how various geometric quantities
vary through a family of Riemannian metrics $g(t)$ evolving via
$\frac{\partial g}{\partial t} =h$ where $h \in s^{2}(TM^{\ast})$.
Much of what is required is explained extremely well in \cite{Top}
and so many proofs are omitted.
\begin{lemma}[Prop 2.3.7 and 2.3.9 in \cite{Top}] \label{L1}
Let $g(s)=g+sh$ be a family of Riemannian metrics on $M$.  Then
   $$ \frac{d}{d s} Ric (g(s))\bigg|_{s = 0} = \frac{1}{2} \nabla^{\ast} \nabla
    (h) - Rm (h, \cdot) + \frac{1}{2} (Ric \cdot h + h \cdot
    Ric) - div^{\ast} div (h) - \frac{1}{2} Hess(tr (h)).$$
Let $R(s)=tr_{g(s)}Ric(g(s))$ be the scalar curvature of $g(s)$ then
$$\frac{d}{d s} R(s) \bigg|_{s = 0} = - \langle h, Ric \rangle + divdiv (h) - \Delta (tr (h)). $$
\end{lemma}
It is necessary to know how geometric quantities associated to
functions $f:(-\epsilon,\epsilon)\times M\rightarrow \mathbb{R} $
vary. The convention we adopt is that $\Delta f =
tr(Hess(f))=-\delta df$, so the Laplacian has negative eigenvalues.
\begin{lemma}[Prop 2.3.10 in \cite{Top}] \label{L2}
Let $g(s)=g+sh$ be a family of Riemannian metrics on $M$ and let
$f(s,x) :(-\epsilon,\epsilon) \times M \rightarrow \mathbb{R}$ be a
family of smooth functions. Then
\begin{align*}
\frac{d}{d s} \Delta (f) \bigg|_{s = 0} &= \Delta_{g_{}} ( \dot{f})
- \frac{1}{2} \langle d (tr (h)), d f \rangle -\langle h,Hess(f)
\rangle -\langle div(h),df \rangle  \\
\frac{d}{d s} | \nabla f|^2 \bigg|_{s = 0} &= - h (\nabla f, \nabla
f) + 2 \langle \nabla f, \nabla \dot{f} \rangle
\end{align*}
\end{lemma}
The final object whose variation must be computed is the Hessian of
the functions $f(s,x)$.
\begin{lemma}
Let $g(s)=g+sh$ be a family of Riemannian metrics on $M$ and let
$f(s,x) : (-\epsilon,\epsilon) \times M \rightarrow \mathbb{R}$ be a
family of smooth functions. Then
$$\frac{d}{d s} Hess(f)\bigg|_{s = 0} = Hess(\dot{f})+\frac{1}{2}(Hess(f)\cdot h+h\cdot Hess(f))
+\frac{1}{2}(\nabla_{\nabla f}h)+div^{\ast}(h(\nabla f,\cdot)).$$
\end{lemma}\label{L3}
\begin{proof}
The Hessian of a function $f$ is given by $Hess(X,Y) =
g(\nabla_{X}\nabla f,Y)$. Hence
\begin{align*}
\frac{d}{ds}Hess(X,Y) &= h(\nabla_{X}\nabla
f,Y)+g(\frac{d}{ds}(\nabla_{X}\nabla f),Y)\\
&= h(\nabla_{X}\nabla f,Y)+g(\frac{d}{ds} (\nabla_{X})\nabla f,Y) +
g(\nabla_{X}\frac{d}{ds}\nabla f,Y).
\end{align*}
By Proposition 2.3.1 in \cite{Top}
$$g(\frac{d}{ds}(\nabla_{X})\nabla f,Y) =\frac{1}{2}[(\nabla_{\nabla f}h)(X,Y)
+(\nabla_{X}h)(Y,\nabla f)-(\nabla_{Y}h)(X,\nabla f)].$$ We also
have
\begin{align*}
g(\nabla_{X}\frac{d}{ds}\nabla f,Y)+h(\nabla_{X}\nabla f,Y)&=
g(\nabla_{X}\nabla \dot{f},Y)-(\nabla_{X}h)(Y,\nabla f)\\
&= Hess(\dot{f})(X,Y)-(\nabla_{X}h)(Y,\nabla f).
\end{align*}
Hence the variation is given by
\begin{align*}
&Hess(\dot{f})(X,Y)+\frac{1}{2}(\nabla_{\nabla
f}h)(X,Y)+\frac{1}{2}[h(\nabla_{X}\nabla f,Y)+h(\nabla_{Y}\nabla
f,X)]\\
& -\frac{1}{2}[(\nabla_{X}h(\nabla f,\cdot))(Y)+(\nabla_{Y}h(\nabla
f,\cdot))(X)],
\end{align*}
and the result follows.
\end{proof}

\begin{lemma}\label{L4} For $h \in s^{2}(TM^{\ast})$, then
$$div_{f}div_{f}(h) = divdiv(h)+h(\nabla f, \nabla f) - 2\langle df,div(h)\rangle-\langle h,Hess(f)\rangle.$$
\end{lemma}
\begin{proof}
From the definition of $div_{f}$ we have $div_{f}(h)(\cdot) =
div(h)(\cdot)-h(\nabla f, \cdot)$ so we need to compute
\begin{align*}div_{f}(div(h)(\cdot)-h(\nabla f, \cdot)) &=
e^{f}div(e^{-f}(div(h)(\cdot)-h(\nabla f, \cdot)) \\
&= divdiv(h)-div(h)(\nabla f) - div(h(\nabla f, \cdot)) +h(\nabla f,
\nabla f).
\end{align*}
The term $div(h(\nabla f, \cdot)) = div(h)(\nabla f)+\langle
h,Hess(f)\rangle$, and so
$$div_{f}div_{f}(h) = divdiv(h)+h(\nabla f, \nabla f) - 2\langle df,div(h)\rangle-\langle h,Hess(f)\rangle.$$
\end{proof}

If the soliton varies by $\delta g=h$, then this induces a variation
in the pair $(f,\tau)$ which is denoted $(\delta f,\delta \tau)$.
\begin{lemma}\label{L5}
  Let \ $(g, f, \tau)$ describe a a Ricci-Soliton and consider a variation
  $g(s)=g+sh$ inducing variations $(\delta f, \delta
  \tau)$ in $f$ and $\tau$.  If $v_{h} = (tr(h)-2\delta f -\frac{2\delta \tau}{\tau}(f-\nu))$, then $v_{h}$
  satisfies
  $$\Delta_{f}(v_{h})+\frac{v_{h}}{2\tau} = div_{f}div_{f}(h)$$
  and
  $$\int_M (- \frac{n \delta \tau}{2 \tau} f + \delta f (1 - f) + f
    \frac{1}{2} tr (h)) e^{- f} d V_g = 0.$$
\end{lemma}
\begin{proof}
Consider the variation in the equation
$$\tau(-2\Delta f+|\nabla f|^{2}-R)-f+n+\nu=0,$$ which yields
$$\delta \tau(-2\Delta f+|\nabla f|^{2}-R)+\tau(-2\delta(\Delta f)+\delta|\nabla f|^{2}-\delta R)-\delta f +\delta \nu = 0.$$
As we are at a soliton, $\delta \nu =0$ and
$R=\frac{n}{2\tau}-\Delta f$. Hence
$$-2\Delta f+|\nabla f|^{2}-R = -\Delta f+|\nabla f|^{2}-\frac{n}{2\tau}=-\Delta_{f}f - \frac{n}{2\tau}.$$
Lemmas \ref{L1} and \ref{L2} imply that
$$\delta(-2\Delta f)=\Delta(-2\delta f)+\langle \nabla tr(h),\nabla f \rangle +2 \langle h,Hess(f) \rangle + 2\langle div(h),df \rangle$$
and
$$-\delta R = \langle h,Ric \rangle -divdiv(h)+\Delta(tr(h))=-\langle h,Hess \rangle+\frac{1}{2\tau}tr(h) -divdiv(h)+\Delta(tr(h)).$$
This yields
$$\frac{\delta \tau}{\tau}(-2\Delta f+|\nabla f|^{2}-R)+ \Delta(tr(h)-2\delta f)+\langle \nabla(tr(h)-2\delta f),\nabla f\rangle+\frac{1}{2\tau}(tr(h)-2\delta f)$$ $$= divdiv(h)+h(\nabla f,\nabla f) - \langle h,Hess \rangle -2\langle div(h),df \rangle.$$
The left hand side is $div_{f}div_{f}(h)$, by Lemma \ref{L5}, and
the right hand side may be written as
 $$\frac{\delta \tau}{\tau}(-2\Delta f+|\nabla f|^{2}-R)+ \Delta_{f}(tr(h)-2\delta f)+\frac{1}{2\tau}(tr(h)-2\delta f).$$
At a Ricci soliton, the scalar curvature $R$ is given by $R=-\Delta
f + \frac{n}{2\tau}$ so
$$\Delta_{f}(f)=\Delta(f)-|\nabla f|^{2} = 2\Delta(f)-|\nabla f|^{2}+R-\frac{n}{2\tau} = \frac{-f}{\tau}+\frac{n}{2\tau} +\frac{\nu}{\tau},$$
whence
$$\Delta_{f}f+\frac{f}{2\tau} = -\frac{f}{2\tau}+\frac{n}{2\tau} +\frac{\nu}{\tau}.$$
Setting $\tilde{f} = -\frac{2\delta \tau}{\tau}(f-\nu)$,
$$\Delta_{f}(\tilde{f})+\frac{\tilde{f}}{2\tau} = \frac{\delta \tau}{\tau}(\frac{f}{\tau}) -\frac{2\delta \tau}{\tau}(\frac{n}{2\tau} +\frac{\nu}{\tau})+\frac{\delta \tau \nu}{\tau^{2}}=\frac{\delta \tau}{\tau}(-2\Delta f+|\nabla f|^{2}-R).$$
Now letting $v_{h} = tr(h)-2\delta f-2\frac{\delta\tau
(f-\nu)}{\tau}$,
$$\Delta_{f}(v_{h})+\frac{v_{h}}{2\tau} = div_{f}div_{f}(h).$$
The second equation is simply the variation in the equation
$$(4\pi\tau)^{\frac{-n}{2}}\int_{M}fe^{-f}dV_{g}=\frac{n}{2}+\nu.$$
\end{proof}

We proceed to the proof of Theorem \ref{T1}:
\begin{proof}
The first variation is given by
$$ \frac{d}{d s} \nu (g) |_{s = 0} = (4 \pi \tau)^{- \frac{n}{2}} \int_M -
     \langle h ,\tau (Ric + Hess (f)) - \frac{1}{2}
     g\rangle e^{- f} d V_g $$
so it is sufficient to compute the variation in the term
$$\tau (Ric) + Hess (f)) - \frac{1}{2}g.$$
This is given by
$$\delta\tau(Ric(g) + Hess(f))-\frac{1}{2}h+\tau(\delta Ric+\delta Hess(f)).$$
Using the previous results
$$\delta Ric =  \frac{1}{2} \nabla^{\ast} \nabla(h) - Rm (h, \cdot) + \frac{1}{2} (Ric \cdot h + h \cdot
    Ric) - div^{\ast} div (h) - \frac{1}{2} Hess(tr (h)),$$
which, using the fact we are at a soliton, yields
\begin{align*}
\delta Ric = & \frac{1}{2} \nabla^{\ast} \nabla(h) - Rm (h, \cdot) -
\frac{1}{2}(Hess(f) \cdot h + h \cdot
    Hess(f))+\frac{h}{2\tau}\\
    & - div^{\ast} div (h) - \frac{1}{2} Hess(tr (h)).
\end{align*}
The variation in the Hessian is given by
$$Hess(\dot{f})+\frac{1}{2}(Hess(f)\cdot h+h\cdot Hess(f))+\frac{1}{2}(\nabla_{\nabla f}h)+div^{\ast}(h(\nabla f,\cdot)).$$
Putting this together  the variation is given by
\begin{align*}
&\delta \tau Ric(g)-\frac{\tau}{2}(\Delta(h)-(\nabla_{\nabla
f}h))-\tau Rm(h,\cdot)-\tau div^{\ast}(div(h)-h(\nabla f,\cdot))\\
&+\frac{\tau}{2}Hess(\frac{2\delta \tau}{\tau}f+2\delta f - tr(h)),
\end{align*}
which may be rewritten as
$$= \delta \tau Ric(g) +\tau (-\frac{1}{2}\Delta_{f}(h)-Rm(h,\cdot)-div_{f}^{\ast}div_{f}(h)-\frac{1}{2}Hess(v_{h})).$$
The result follows on taking $\frac{\delta \tau}{\tau}=C(h,g).$
\end{proof}

Taking $f$ constant  recovers:
\begin{corollary}[Einstein Case \cite{CHI}]
Let $g$ be an Einstein metric with $Ric(g)=\frac{1}{2\tau}g$. For
$h\in s^2(TM^{\ast})$, consider variations  $g(s)=g+sh$. Then the
second variation of the $\nu$-energy at is
$$\frac{d^2}{ds}\bigg|_{s=0} \nu(g(s)) = \frac{\tau}{vol(g)}\int_{M}\langle h,Nh\rangle dV_{g},$$
where $N$ is given by
$$N(h)=-\frac{1}{2}\nabla^{\ast}\nabla(h)+Rm(h,\cdot)+div^{\ast}div(h)+\frac{1}{2}Hess(v_{h})-\frac{g}{2n\tau vol(g)}\int_{M}tr(h)dV_{g}.$$
Here $v_h$ is the solution of
$$\Delta(v_{h})+\frac{1}{2\tau}v_{h}=divdiv(h).$$
\end{corollary}

\section{The instability of K\"ahler-Ricci solitons}
We first establish the following simple Lemma comparing curvature
operators on a Kahler manifold:
\begin{lemma}\label{curv}
  Let $(M, g, J)$ be a Kahler manifold. If $\sigma \in \mathcal{A}^{(1,1)}
  (M)$, then $2 R m (\sigma_J, \cdot) - \mathcal{R} (\sigma)_J = 0.$
  \end{lemma}

\begin{proof}
  We calculate using an adapted orthonormal basis $\{e_{i}, J e_{i} \}$. \ From
  the definition of $\mathcal{R}$, the Bianchi identity and the $J$-invariance of the curvature
tensor;
\begin{align*}
\mathcal{R} (\sigma) (X, J Y) &= \sum_{i, j} R (e_i, e_j, X, J Y)
\sigma (e_i, e_j) \\
&= - \sum_{i, j} \left( R (X, e_i, e_j, J Y) + R (e_j, X, e_i, J Y)
\right) \sigma (e_{i,} e_j) \\
&= \sum_{i, j} \left( R (X, e_i, J e_j, Y) + R (e_j, X, J e_i, Y)
\right) \sigma (e_i, e_j) \\
&= \sum_{i, j} - 2 R (e_i, X, J e_j, Y) \sigma (e_i, e_j)\\
&= \sum_{i, j} 2 R (e_i, X, J e_j, Y) \sigma (e_i, J (J e_j)) \\
&= 2 R m (\sigma_J, \cdot) (X, Y).
\end{align*}
\end{proof}

The proof of the main result will be modeled on the Kahler-Einstein
case which is outlined in \cite{CHI}.
\begin{proposition}
Let $(M,g,J)$ be a K\"ahler-Einstein manifold with $\dim
H^{(1,1)}(M) \geq 2$. Then $g$ is linearly unstable.
\end{proposition}

\proof Choose a trace-free harmonic $(1,1)$-form $\sigma \in
\mathcal{H}^{(1,1)}(M)$ that induces a perturbation $\sigma_{J} =
\sigma (\cdot, J \cdot)$. As the complex structure is parallel
$div(\sigma_J) = 0$ and $(\nabla^{\ast} \nabla \sigma_J) =
(\nabla^{\ast} \nabla \sigma_{})_J$. We can use a Weitzenbock
formula on 2-forms to compare the rough Laplacian with the Hodge
Laplacian $\Delta_H=-(d+\delta)^{2}$. The Weitzenbock formula for an
Einstein metric with Einstein constant $\frac{1}{2 \tau}$ is
$$ -\Delta_H = \nabla^{\ast} \nabla - \mathcal{R} + \frac{1}{\tau}id $$
where $\mathcal{R}$ is the curvature operator for 2-forms. Our
formula for the variation becomes
\begin{align*}
2 N (\sigma_J) &= \Delta_H (\sigma)_J - \mathcal{R} (\sigma)_J +
\frac{1}{\tau} \sigma_J + 2 R m (\sigma_J, \cdot) +
Hess(v_{\sigma_J})\\
&= - \mathcal{R} (\sigma)_J + \frac{1}{\tau}
   \sigma_J + 2 R m (\sigma_J, \cdot) + Hess(v_{\sigma_J}).
\end{align*}
It follows from Equation \ref{E1} that the function $v_{\sigma_J}$
is an eigenfuction of the Laplacian. Recall that the first non-zero
eigenvalue of the Laplacian on an Einstein manifold with Einstein
constant $c >0$ satisfies
$$
\lambda_1 \leq \frac{-n}{n-1}c
$$
by the famous Lichnerowicz bound. Hence $v_{\sigma_J} = 0$. From
Lemma \ref{curv}, $2Rm(\sigma_J, \cdot) - \mathcal{R} (\sigma)_J =
0$ so we have exhibited an eigentensor for $N$ with eigenvalue
$\frac{1}{2 \tau}$. Hence metrics of this type are
unstable.
\endproof

Considering the divergence operator $\delta :
\Omega(M)^{k}\rightarrow \Omega(M)^{k-1}$
 as the adjoint of the exterior derivative on forms then, with our conventions,
 $div(\sigma)=-\delta(\sigma)$ for any $k$-form $\sigma$. Set
$$\delta_{f}(\sigma) = e^{f}\delta(e^{-f}\sigma)= \delta(\sigma)+\iota_{\nabla f}
\sigma \text{ and }
\Delta_{f,H}(\sigma)=-(d\delta_{f}+\delta_{f}d)(\sigma).$$ We will
refer to $\Delta_{f,H}$ as the \emph{twisted Laplacian}. Forms
$\sigma$ satisfying $\Delta_{f,H}(\sigma) = 0$ are said to be
\emph{twisted harmonic forms}. Modifications of the Hodge Laplacian
similar to this are also found in the seminal work of Witten
\cite{Witten} on the Morse inequalities. The following Lemma gives
some crucial properties of this Laplacian:
\begin{lemma}
The operator $\Delta_{f,H}$ has the following properties:
\begin{enumerate}
\item $\Delta_{f,H}=\Delta_{H}-\mathcal{L}_{\nabla f}$.
\item $\Delta_{f,H}$ preserves the decomposition into $(p,q)-forms$.
\item $\Delta_{f,H}$ satisfies a Weitzenbock identity for $2$-forms
$\sigma$:
$$\Delta_{f}(\sigma)-\Delta_{f,H}(\sigma) =-\mathcal{R}(\sigma)+\frac{1}{\tau}\sigma .$$
\end{enumerate}
\end{lemma}
\begin{proof}
From the definition we have $$\Delta_{f,H}(\sigma) =
\Delta_{H}(\sigma)-(d\circ \iota_{\nabla f}+\iota_{\nabla f} \circ
d)\sigma.$$ Using  Cartan's magic formula this yields
$$\Delta_{f,H}(\sigma) =  \Delta_{H}(\sigma)-\mathcal{L}_{\nabla f}(\sigma).$$ The second claim follows from the
 fact that $\nabla f$ is a holomorphic vector field. For the Weitzenbock formula recall that for 2 forms
\begin{align*}
-\nabla^{\ast}\nabla(\sigma)-\Delta_{H}(\sigma)&=-\mathcal{R}(\sigma)+Ric
\cdot \sigma +\sigma \cdot Ric \\
&= -\mathcal{R}(\sigma) - (Hess(f) \cdot \sigma +\sigma \cdot
Hess(f)) +\frac{1}{\tau}\sigma,
\end{align*}
where $\mathcal{R}$ is the curvature operator for $2$-forms. A
simple computation yields
$$\mathcal{L}_{\nabla f}(\sigma)-(\nabla_{\nabla f})(\sigma) = (Hess(f) \cdot \sigma +\sigma \cdot Hess(f)).$$
Hence
$$(\Delta_{f}-\Delta_{f,H})(\sigma) = -\mathcal{R}(\sigma)+\frac{1}{\tau}\sigma.$$
\end{proof}

To conclude the proof of Theorem \ref{main}  is presented:

\begin{proof}
The first step is to choose a twisted harmonic form $\sigma \in
H^{(1,1)}(M)$. If we denote the space of twisted harmonic forms by
$\mathcal{H}_{f}$ then we have the natural map
$$\pi:\mathcal{H}_{f} \rightarrow H^{2}(M) \text{ given by } \pi(\sigma)=[\sigma].$$
The proof that this map is an isomorphism follows that of the usual
Hodge theorem except that the `energy' in our case is given by
$$\int_{M}\|\sigma\|^{2}e^{-f}dV_{g}.$$
The twisted Laplacian also preserves the $(p,q)$-decomposition of
forms so we get the required isomorphism of vector spaces
$$\mathcal{H}^{(1,1)}_{f}\cong H^{(1,1)}(M).$$
The hypothesis on the dimension of $H^{(1,1)}(M)$ means that there
exist $\theta_{1}, \theta_{2} \in \mathcal{H}_{f}^{(1,1)}$ and
$\lambda, \mu \in \mathbb{R}$, not both zero, satisfying
$$\langle \lambda \theta_{1}+ \mu \theta_{2},\rho \rangle_{f} =
 \int_{M} \langle \lambda \theta_{1}+ \mu \theta_{2},\rho \rangle e^{-f}dV_{g} =0,$$
where $\rho$ is the Ricci form. Choosing $\sigma = \lambda
\theta_{1}+ \mu \theta_{2}$ and computing gives $$\langle
N(\sigma_{J}),\sigma_{J}\rangle_{f} =
\frac{1}{2\tau}\|\sigma_{J}\|^{2} > 0,$$ and the theorem follows.
\end{proof}

Corollary \ref{minor} follows immediately from this theorem.

\bigskip

\textit{Remark:} In the recent work \cite{CMZ} the authors consider
the operator
$$\hat{\mathcal{L}} = \frac{1}{2}\Delta_{f}+Rm(h,\cdot).$$ They prove that the Ricci
form $\rho$ is twisted harmonic and that it satisfies
$\hat{\mathcal{L}}(Ric)=\frac{1}{2\tau}Ric$. Our proof here can be
phrased as showing that the multiplicity of this eigenspace is at
least $\dim H^{(1,1)}(M)$. One may then apply their Proposition 3.1
to conclude that K\"ahler-Ricci solitons are linearly unstable.

\bigskip
It would be most interesting to determine whether the perturbations
$\sigma_J$ are in fact eigentensors of the operator $N$. It is clear
that if $v_{\sigma_J}=0$ then $\sigma_J$ is an eigentensor. We
observe that the recent work of Futaki-Sano \cite{FutakiSano} and Ma
\cite{Ma} gives a spectral gap for the  Bakry-\'{E}mery
Laplacian when the Bakry-\'{E}mery Ricci curvature is bounded below.
They show:
\begin{theorem}[Ma, Futaki-Sano]
Let $\phi$ be a smooth function on $M$ and suppose that the
Bakry-\'{E}mery Ricci curvature, defined as
$$
Ric_{\phi} = Ric(g) + Hess(\phi),
$$
satisfies the bound
$$
Ric_{\phi} \geq c,
$$
for some number $c > 0$. Then the first non-zero eigenvalue
$\lambda_1$ of the Bakry-\'{E}mery Laplacian $\Delta_{\phi}:= \Delta
- \nabla_{\nabla\phi}$ satisfies
$$
\lambda_1 \leq -c.
$$
\end{theorem}
Clearly this bound is weaker that the Lichnerowicz bound in the
Einstein case. Hence one may not deduce that $v_{\sigma_J} = 0$ for
the unstable variations of Theorem \ref{main}. However, if
$v_{\sigma_J} \neq 0$ such a soliton would prove that the above
estimate is in some sense sharp. We suspect however that $v_{\sigma_j} = 0$ and hence
that these variations are eigentensors of $N$. Implicit in the calculation of the second variation formula is the fact that the potential function $f$ (once normalised) is an eigenfunction of $\Delta_{f}$ with eigenvalue $-\frac{1}{\tau}$ (Futaki and Sano also prove this). It would be interesting to know if $\Delta_{f}$ has any eigenvalues in the interval $(-\frac{1}{\tau},-\frac{1}{2\tau}]$ and the authors are currently
numerically investigating this problem for the Koiso-Cao and Wang-Zhu solitons.

\end{document}